\documentclass[12pt]{article}
\usepackage{a4wide,amsmath,amssymb,amsthm,graphicx,color,comment}
\definecolor{orange}{rgb}{0.5,0.4,0.0}
\newtheorem{theorem}{Theorem}[section]
\newtheorem{corollary}[theorem]{Corollary}
\newtheorem{proposition}[theorem]{Proposition}
\newtheorem{lemma}[theorem]{Lemma}
\newtheorem{remark}[theorem]{Remark}
\newcommand\skiplinear[1]{#1}
\newcommand\skipgradient[1]{}
\newcommand\MRIdetails[1]{}
\newcommand\ul[1]{\underline{#1}}
\newcommand\uld[1]{\underline{d#1}}
\newcommand\DD{\mathcal{D}}
\newcommand\calA{\mathcal{A}}
\newcommand\BB{\mathcal{B}}
\newcommand\HH{\mathcal{H}}
\newcommand\Op{\mathcal{N}}
\newcommand\Proj{\mathcal{P}}
\newcommand\ellnew{\ell}
\newcommand\coilsens{\omega}
\newcommand\dualmap{\mathfrak{j}}
\newcommand\hrefprime{\mathcal{Q}_{\text{ref}}}
\newcommand\pot{\xi}
\newcommand\init{\eta}
\newcommand\vecr{\vec{x}}
\newcommand\contr{p}
\newcommand\Contr{P}
\newcommand\ulB{\underline{B}}
\newcommand\ulC{\underline{C}}
\newcommand\ulH{\underline{H}}
\newcommand\ulQ{\underline{Q}}
\newcommand\ulPsi{\underline{\Psi}}
\newcommand\Ncont{N_p}
\newcommand\Ncoil{N_{rc}}
\begin{document}
\title{Linearized uniqueness of space dependent coefficients in a non-autonomous evolution equation from non-local observations}
\author{Barbara Kaltenbacher\footnote{
Department of Mathematics,
Alpen-Adria-Universit\"at Klagenfurt.
barbara.kaltenbacher@aau.at}
}

\maketitle
\begin{abstract}
In this paper, we consider identification of space dependent coefficients in a time dependent PDE, which is non-autonomous, due to a bilinear control term.
We prove linearized uniqueness from time trace observations, in particular also considering non-local observations in the form of weighted integrals of the state over he spatial domain.
The result is obtained in an abstract setting and applied, first of all to the identification of a potential in a diffusion equation for illustration purposes, second to the reconstruction of equilibrium magnetization, relaxation rates and field inhomogeneity as space dependent quantities in model based qunatitative magnetic resonance imaging with the Bloch-Torrey equation, which constitutes the real world application motivating this study.
\end{abstract}

\noindent
\textbf{key words:} nonautonomous diffusion equation; multi-coefficient identification; uniqueness; Bloch-Torrey equation.\\
\textbf{AMS Subject Classification 2020:} 35R30, 35K20, 92C55

\section{Introduction}\label{sec:intro}
This paper provides an approach of proving uniqueness of space dependent coefficients in an evolutionary PDE 
\begin{equation}\label{PDE}
\DD_t u(t) +(\calA+ \BB(t)+\HH(\theta))u(t)=f(t) \quad t\in(0,T)\qquad u(0)=u_0(\theta)
\end{equation}
from time trace observations
\begin{equation}\label{obs}
y(t)=\mathcal{C}u(t)  \qquad t\in(0,T).
\end{equation}

It is motivated by imaging applications that are governed by evolutionary PDEs, where the imaging quantities are space dependent coefficients $\theta(x)$ and the imaging process can be steered by the action of a time dependent control operator $\BB(t)$ on the PDE state $u(t,x)$.
While $\DD_t$ is a time derivative operator (not necessarily of first, potentially even of fractional order), the operator $\calA$ is a spatial differential operator modelling diffusion on a domain $\Omega\subseteq\mathbb{R}^d$. (More details on the operators and the function space setting for \eqref{PDE} and \eqref{obs} will be given in Section~\ref{sec:setting} below.)

Motivated by applications, we are particularly interested in averaging, thus non-local  observations of a weighted spatial integral of the state
\begin{equation*}
y(t)=\mathcal{C}u(t)=\int_\Omega \coilsens(x)\,u(t,x) \, dx =\langle \coilsens,\,u(t)\rangle_{L^2(\Omega)},
\end{equation*}
cf. \eqref{obs_BT}, \eqref{C_avrg}, which makes the inverse problem of reconstructing space dependent coefficient unamenable to classical uniqueness techniques such as Carleman estimates \cite{Klibanov:1992,YamamotoCarleman:2025} or spectral methods 
\cite{KianLiLiuYamamoto:2021,Pierce:1979}, that are commonly used in case of interior or boundary observations. 
Identification of spatially variable coefficients from time traces of weighted averages is by far less well studied than the mentioned interior or boundary observation setting; to the best of the author's knowledge, existing literature is targeted at time dependent coefficients rather than space dependent ones, see, e.g., \cite{HazaneeLesnicIsmailovKerimov:2019,KerimovNazimIsmailov:2012,VanBockstalKhompysh:2026}.
Note that as compared to, e.g., final time observations inside the spatial domain, the fact that the observation is taken into an orthogonal direction of variability in \eqref{obs}, generally makes the inverse problem harder to tackle.  

\medskip

Concretely, we have in mind model based quantitative magnetic resonance imaging MRI, where one aims to map 
the space dependent relaxation rates $R_z$, $R_\perp$ and the equilibrium magnetization $M^{eq}$ 
\MRIdetails{containing the spin density} 
in the Bloch-Torrey PDE
\begin{equation}\label{eqn:bloch-torrey}
\begin{aligned}
&    \frac{d}{dt} \vec M(t, \vecr) + \gamma B^1(t, \vecr)\times\vec M(t, \vecr)    
+ \text{diag}(R_\perp(\vecr),R_\perp(\vecr),R_z(\vecr)) ({\vec M}(t, \vecr)-{\vec M}^0(\vecr))
\\&\hspace*{1cm}
    -\nabla\cdot\Bigl(D(\vecr) \, \nabla \vec M(t, \vecr)\Bigr)
=0\quad (t,\vecr)\in (0,T)\times\Omega, \quad
{\vec M}(0, \vecr)={\vec M}^0(\vecr)\quad \vecr\in \Omega
\end{aligned}
\end{equation}
with 
\begin{equation}\label{eqn:bloch-torrey-defs}
\begin{aligned}
&B^1(t, \vecr)=
\begin{pmatrix} p(t)c^+_{x}(\vecr)\\p(t)c^+_{y}(\vecr)\\
B^0 + \delta B^0(\vecr) \skipgradient{+\vecr \cdot \vec{g}(t)}
\end{pmatrix},
\quad 
{\vec M}^0(\vecr)=   \begin{pmatrix} 0 \\ 0 \\ M^{eq}(\vecr)\end{pmatrix}.
\\
\end{aligned}
\end{equation}
The diffusion operator $\Delta_D:=\nabla\cdot\Bigl(D(\vecr) \, \nabla \vec \cdot \Bigr)$ acts component wise on $\vec{M}$ and is equipped with homogeneous Dirichlet boundary conditions.
Also the field inhomogeneity $\delta B^0(\vecr)$ carries important information and an analogous term is relevant for model based quantitative susceptibility mapping \cite{QSM}.

Here time dependence 
of the radio frequency pulse $p$ 
\skipgradient{and/or of the gradient field $\vec{g}$} 
is essential in order to guarantee a sufficient excitation and sampling for the recovery of the imaging quantities $R_z(\vecr)$, $R_\perp(\vecr)$, $M^{eq}(\vecr)$, $\delta B_0(\vecr)$
and a substantial amount of research has thus already gone into the design of pulse sequences, cf., e.g., \cite{GrafSoellradlAignerRundStollberger:2022} 
and the references therein.
Along these lines, the criteria we derive on uniqueness in  Proposition~\ref{prop:lin-uniqueness}, Theorem~\ref{thm:lin-uniqueness}, Corollary~\ref{cor:BT} below, can give a general guideline on the choice of $p$ 
\skipgradient{(and $\vec{g}$)}  
to achieve a desired spatial resolution of $R_z$, $R_\perp$, $M^{eq}$, $\delta B_0$.

The typical observations are obtained as averages of the transversal magnetic field $(M_x,M_y)$, taken in parallel at ${\Ncoil}$ receive coils, with coil sensitivities $(\coilsens^j_x,\coilsens^j_y)$\MRIdetails{The fact that both components of the transversal magnetic field can be observed is achieved by qudrature demodulation, cf. Nishimura page 77.} 
\MRIdetails{The vector $(\coilsens^j_x,\coilsens^j_y)$ plays the role of the (transversally oriented) magnetic flux density per unit current of the $j$th receive coil; Faraday's law implies that the electromotive force (which is in fact work per charge rather than work per length) equals the negative time derivative of the magnetic flux through the coil, which in its turn by the reciprocity principle equals the inner product of the coil magnetic flux density per unit current with the magnetization.}
\begin{equation}\label{obs_BT}
y(t)=(\langle \coilsens^j_x,M_x(t)\rangle_{L^2(\Omega)},\,\langle \coilsens^j_y,M_y(t)\rangle_{L^2(\Omega)}), \quad j\in\{1,\ldots,{\Ncoil}\}.
\end{equation}
In view of the lower dimensionality of these observations, uniqueness of the imaging quantities $R_z$, $R_\perp$, $M^{eq}$, $\delta B_0$,  as functions of 3D space cannot be expected to hold on all of, e.g., a Lebesgue space $L^r(\Omega)$, $r\in[1,\infty]$. We therefore also aim to characterize the subspace of $L^r(\Omega)$ on which \eqref{obs_BT} uniquely determines $R_z$, $R_\perp$, $M^{eq}$, $\delta B_0$.

This work follows up on \cite{MRI-uniqueness}, where we have proven linearized uniqueness with very particular choices of $p(t)$ in the control term $(p(t)c^+_{x}(\vecr),p(t)c^+_{y}(\vecr))^T$.
In view of the importance of controlling the data acquisition by a proper excitation, the result provided here (Corollary~\ref{cor:BT}) contributes to laying the basis for a further mathematical optimization of pulse sequence design.

\medskip

Returning to the general setting \eqref{PDE}, \eqref{obs}, our aim is to provide a framework for proving 
linearized uniqueness, 
that is of injectivity of 
the derivative $\mathbb{F}'(u_{\text{ref}},\theta_{\text{ref}})$ of 
the forward operator
\begin{equation}\label{F}
\mathbb{F}:(u,\theta)\mapsto \left(\begin{array}{l}
t\mapsto\DD_t u(t) +(\calA+ \BB(t)+\HH(\theta))u(t)-f(t)\\
u(0)-u_0\\
t\mapsto\mathcal{C}u(t)-y(t)
\end{array}\right)
\end{equation}
at a properly chosen reference point $(u_{\text{ref}},\theta_{\text{ref}})$ under certain conditions on the control $\BB$, that derive from the injectivity proof.

Linearized uniqueness -- though not necessarily implying full uniqueness, since in an ill-posed setting the inverse function theorem does not apply -- is important for various reasons.
First of all, it entails applicability of Newton's method and regularized versions thereof, see e.g., \cite{BakKok04,KNSbook:2008} and the references therein. 
Secondly, it implies that regularity assumptions required for convergence rates (so-called source conditions) are generically satisfied, see, e.g., \cite{Flemming2018}.
Thirdly, criteria for the choice of the control in the sense of an optimal experimental design \cite{Alexanderian2021,Koerkeletal,Pukelsheim} are often based on the linearization of the forward operator and usually targeted at its injectivity. 

\medskip

The remainder of this paper is organized as follows.
In Subsection~\ref{sec:setting}, we provide more details on the function space setting of \eqref{PDE}, \eqref{obs}. 
Section~\ref{sec:uniqueness} is concerned with deriving (linearized) uniqueness criteria in this general setting. On the one hand, this leads to a linear independence condition on fundamental ODE solutions determined by the control, thus giving a guideline on control design. On the other hand, for any prescribed control, we spell out maximal coefficient spaces on which the linearized forward operator is injective. 
This abstract setting as well as the conditions for the uniqueness results are  first of all illustrated by a scalar model problem in Section~\ref{sec:diffusion}.
Finally, in Section~\ref{sec:BT}, we return to the motivating example \eqref{eqn:bloch-torrey}, \eqref{obs_BT} and verify the conditions required for the results in Section~\ref{sec:uniqueness}.

\subsection{Setting}\label{sec:setting}
We denote the state space by $U\subseteq L^p(0,T;V)$ and understand \eqref{PDE} to hold in an $L^p(0,T;W^*)$ sense for reflexive  spaces $V$, $W$ and $p\in[1,\infty]$.
Correspondingly, in \eqref{PDE} and \eqref{obs}, with a reference parameter $\theta_{\text{ref}}$,
\begin{itemize}
\item $\DD_t:U\to L^p(0,T;W^*)$ is a time differential operator, not necessarily of first order; it could as well be of higher or fractional order or a linear combination thereof; 
\item $\calA+\HH(\theta_{\text{ref}})\in L(V,W^*)$ is a spatial differential operator (e.g., a second order elliptic one) that is assumed to give rise to an eigensystem 
$(\lambda_\ell,\mathbb{E}_\ell,\Proj_\ell)_{\ell\in\mathbb{N}}$ with eigenvalues $\lambda_\ell\in\mathbb{C}$, 
eigenspaces\footnote{which we do not necessarily assume to be finite dimensional here} $\mathbb{E}_\ell=\text{ker}(\calA+\HH(\theta_{\text{ref}})-\lambda_\ell)$ and eigenprojections $\Proj_\ell=\text{Proj}_{\mathbb{E}_\ell}\in L(V,V)$ so that  
\begin{equation}\label{APlambdaP}
(\calA+\HH(\theta_{\text{ref}}))v_\ell=\lambda_\ell v_\ell, \quad v_\ell\in \mathbb{E}_\ell, \quad \ell\in\mathbb{N};
\end{equation}
Note that the identity \eqref{APlambdaP} holds in $W^*$ and implies well-definedness of $\Proj_\ell$ on $W^*$.
\item For any $t\in (0,T)$, $\theta\in\Theta$, the operators $\BB(t)\in L(V,W^*)$ and $\HH(\theta)\in L(V,W^*)$ act on the state $u$ in a pointwise in time way; 
\item So does the operator $\mathcal{C}\in L(V,Y)$; 
\item Correspondingly, we consider the data $y\in L^p(0,T;Y)$.
\end{itemize}
Since we use an all-at-once formulation of the parameter identification problem, our analysis is not necessarily tied to typical function space settings of parabolic, hyperbolic or other evolutionary PDE theory.
Yet, the use of a Bochner space setting is clearly very natural in the context of time dependent models and observations. 

\section{\skiplinear{Linearized} uniqueness for the abstract problem
\eqref{PDE}, \eqref{obs}
}\label{sec:uniqueness}
The linearization $\mathbb{F}'(u_{\text{ref}},\theta_{\text{ref}})$ of the forward operator $\mathbb{F}:U\times\Theta\to L^p(0,T;W^*)\times U_0\times L^p(0,T;Y)$ defined by \eqref{F} at some reference state-parameter pair $(u_{\text{ref}},\theta_{\text{ref}})$ can straightforwardly be derived as 
\begin{equation}\label{Fprime}
\begin{aligned}
&\mathbb{F}'(u_{\text{ref}},\theta_{\text{ref}}):U\times\Theta\to L^p(0,T;W^*)\times U_0\times L^p(0,T;Y)\\
&(\uld{u},\uld{\theta})\mapsto \left(\begin{array}{l}
t\mapsto\DD_t \uld{u}(t) +(\calA+ \BB(t)+\HH(\theta_{\text{ref}}))\uld{u}(t)+\HH'(\theta_{\text{ref}})\uld{\theta}\, u_{\text{ref}}(t)\\
\uld{u}(0)-u_0'(\theta_{\text{ref}})\uld{\theta}\\
t\mapsto\mathcal{C}\uld{u}(t)
\end{array}\right).
\end{aligned}
\end{equation}
It is indeed a G\^{a}teaux / Fr\'{e}chet derivative of $\mathbb{F}$ provided $\HH:\Theta\to L(V,W^*)$ and $u_0:\Theta\to U_0$ are G\^{a}teaux / Fr\'{e}chet differentiable. In the examples detailed in Sections~\ref{sec:diffusion} and \ref{sec:BT}, both are linear and thus Fr\'{e}chet differentiability holds.

Our goal is to prove the implication 
\begin{equation}\label{Fprimeinjective}
\mathbb{F}'(u_{\text{ref}},\theta_{\text{ref}})(\uld{u},\uld{\theta})=0 \ \Rightarrow  \ (\uld{u},\uld{\theta})=0
\end{equation}

We choose the reference point $(\theta_{\text{ref}},u_{\text{ref}})$ such that the state $u_{\text{ref}}$ is space-time separable, (cf., e.g. \cite{MRI-uniqueness} for the Bloch-Torrey context of MRI) such that for some time dependent function $\psi$ and some operator $\hrefprime$, the following holds. 
\begin{equation}\label{spacetimeseparable}
\begin{aligned}
&(\HH'(\theta_{\text{ref}})u_{\text{ref}})(t)=\psi(t)\hrefprime, \quad 
\psi(t)\in L(W^*,W^*),\ t\in(0,T),\quad \hrefprime\in L(\Theta,W^*)\\ 
\end{aligned}
\end{equation}

Moreover, we assume the control operator and the time dependent part of the reference state to 
commute with the eigenprojections in such a way that 
\begin{equation}\label{commBpsi}
\Proj_\ell[\BB(t) v]=\hat{\BB}_\ellnew(t)\,\Proj_\ell v, \quad
\Proj_\ell[\psi(t)v]=\hat{\psi}_\ellnew(t)\,\Proj_\ell v, \qquad v\in V
\end{equation}
for some operators $\hat{\BB}_\ellnew(t)$, $\hat{\psi}_\ellnew(t)$ $\in L(\mathbb{E}_\ell,\mathbb{E}_\ell)$, possibly different from $\BB(t)$, $\psi(t)$, and possibly depending on $\ellnew$.
In the examples below, this can be verified by using the fact that there, $\BB(t)$ and $\psi(t)$ just act as multiplication operators.

Projecting onto the eigenspaces, we can thus write the model part of the premise of \eqref{Fprimeinjective} as a sequence of ODEs in terms of the eigenspace components
\[
\begin{cases}
\Bigl(\DD_t+\lambda_\ell+\hat{\BB}_\ellnew(t)\Bigr)\Proj_\ell\uld{u}(t)+\hat{\psi}_\ellnew(t)\,\Proj_\ell[\hrefprime\uld{\theta}]=0, \quad t\in(0,T)\\  
\Proj_\ell\uld{u}(0)=\Proj_\ell[u_0'(\theta_{\text{ref}})\uld{\theta}]
\end{cases} \qquad \ell\in\mathbb{N}.
\]
At this point, previous uniqueness proofs often resorted to Laplace transform techniques; this is here impeded by the fact that the ODEs are non-autonomous due to the time dependent control term.
We therefore proceed in an alternative way\footnote{but will revisit the Laplace transform toolbox in the examples sections}, by defining,
for each $\ell\in\mathbb{N}$, the operator valued functions $\Psi_\ell,\,\Psi_\ell^0:[0,T)\to
L(\mathbb{E}_\ell,\mathbb{E}_\ell)
$ as the solutions to the initial value problem  
\begin{equation}\label{Psi_ell}
\begin{aligned}
&\Bigl(\DD_t+\lambda_\ell+\hat{\BB}_\ellnew(t)\Bigr)\Psi_\ell(t)+\hat{\psi}_\ellnew(t)=0, \ t\in(0,T),\quad \Psi_\ell(0)=0\\  
&\Bigl(\DD_t+\lambda_\ell+\hat{\BB}_\ellnew(t)\Bigr)\Psi_\ell^0(t)=0, \ t\in(0,T),\quad \Psi_\ell^0(0)=\text{id}_{\mathbb{E}_\ell},
\end{aligned}
\end{equation}
which allows us to express 
$\Proj_\ell\uld{u}(t)$ in terms of $\Proj_\ell[\hrefprime\uld{\theta}]$, $\Proj_\ell[u_0'(\theta_{\text{ref}})\uld{\theta}]$
\begin{equation}\label{bfroma}
\Proj_\ell\uld{u}(t)=
\Psi_\ell(t) \Proj_\ell[\hrefprime\uld{\theta}] + \Psi_\ell^0(t)\Proj_\ell[u_0'(\theta_{\text{ref}})\uld{\theta}], 
\quad t\in(0,T),\quad \ell\in\mathbb{N}.
\end{equation}
In the examples below, the eigenspaces are finite dimensional and so $\Psi_\ell(t)$, $\Psi_\ell^0(t)$ can simply be defined as multiplication operators with scalar (or matrix valued) functions.

Let us for the moment consider the case that we also have a commutation relation between these ODE solutions with the observation operator, in the sense that 
\begin{equation}\label{commC}
\mathcal{C}\Psi_\ell(t)=\hat{\Psi}_\ell(t) \hat{\mathcal{C}}_\ellnew, \quad
\mathcal{C}\Psi_\ell^0(t)=\hat{\Psi}_\ell^0(t) \hat{\mathcal{C}}_\ellnew^0, \quad 
t\in(0,T),\quad \ell\in\mathbb{N}
\end{equation}
for some operators $\hat{\mathcal{C}}_\ellnew$, $\hat{\mathcal{C}}^0_\ellnew$ $\in L(V,V)$, $\hat{\Psi}_\ell(t)$, $\hat{\Psi}_\ell^0(t)$ $\in L(V,Y)$.

Using this and \eqref{bfroma}, the observation part of the premise of \eqref{Fprimeinjective} can be written as 
\[
0=\mathcal{C}\uld{u}(t)= 
\sum_{\ell\in\mathbb{N}} \Bigl(\hat{\Psi}_\ell(t) \hat{\mathcal{C}}_\ellnew \Proj_\ell[\hrefprime\uld{\theta}] + \hat{\Psi}_\ell^0(t) \hat{\mathcal{C}}^0_\ellnew \Proj_\ell[u_0'(\theta_{\text{ref}})\uld{\theta}]\Bigr)
\quad t\in(0,T),
\]
where we have used linearity and continuity of $\mathcal{C}$ to exchange the infinite sum with application of $\mathcal{C}$.

A key assumption on the time dependent part $\psi$ of the reference state $u_{\text{ref}}$ (and on the control $\BB$) is now that 
the functions 
$\{\hat{\Psi}_\ell,\,\hat{\Psi}_\ell^0:(0,T)\to L(\mathbb{E}_\ell,\mathbb{E}_\ell)\, :\, \ell\in\mathbb{N}\}$
defined by \eqref{commBpsi}, \eqref{Psi_ell}, \eqref{commC} are linearly independent; by this we mean that for $v_\ell,v_\ell^0\,\in \mathbb{E}_\ell$, $\ell\in\mathbb{N}$
\begin{equation}\label{linearindependence}
\Bigl(\sum_{\ell\in\mathbb{N}}\bigl(\hat{\Psi}_\ell v_\ell+\hat{\Psi}_\ell^0 v_\ell^0\bigr)=0\quad \forall t\in(0,T)\Bigr) \ \Rightarrow \
\Bigl( v_\ell=0, \ v_\ell^0=0 \quad \forall \ell\in\mathbb{N}\Bigr),
\end{equation} 
which allows us to conclude 
\begin{equation}\label{obs0}
0= \hat{\mathcal{C}}_\ellnew\, [\Proj_\ell \hrefprime\uld{\theta}]
\ \text{ and } \
0= \hat{\mathcal{C}}^0_\ellnew\, [\Proj_\ell u_0'(\theta_{\text{ref}})\uld{\theta}]
\quad \ell\in\mathbb{N}.
\end{equation}

This implies $\uld{\theta}=0$, provided $\hat{\mathcal{C}}$ satisfies the following completeness property with respect to the parameter space
\begin{equation}\label{cond:Xpar_gen}
\Theta^{par}\,\cap\,\bigcap_{\ell\in\mathbb{N}} 
\text{ker}\bigl(\hat{\mathcal{C}}_\ellnew \Proj_\ell \hrefprime\bigr)\cap 
\text{ker}\bigl(\hat{\mathcal{C}}^0_\ellnew \Proj_\ell u_0'(\theta_{\text{ref}})\bigr)
=\{0\}.
\end{equation}

By \eqref{bfroma} we also have $\uld{u}=0$.

Thus we have proven the following.
\begin{proposition}\label{prop:lin-uniqueness}
Let $u_{\text{ref}}$ be defined such that \eqref{spacetimeseparable} holds, with $\psi$ and $\BB$ chosen such that \eqref{commBpsi}, \eqref{commC} and \eqref{linearindependence} hold.
Then on any subspace $\Theta^{par}$ of $\Theta$ satisfying \eqref{cond:Xpar_gen}, the linearized forward operator $\mathbb{F}'(u_{\text{ref}},\theta_{\text{ref}}):\Theta^{par}\to L^p(0,T;Y)$ is injective.
\end{proposition}

To relax the linear independence assumption and dispose of \eqref{commC}, we write the linearized observation equation $\mathcal{C}\uld{u}=0$ with \eqref{bfroma} as $\Op\uld{\theta}=0$ with the operator 
\[
\begin{aligned}
&\Op:\Theta\to L^2(0,T;Y), \quad \\ &
\theta\mapsto \Bigl(t\mapsto
\sum_{\ell\in\mathbb{N}}
\bigl(\mathcal{C}\Psi_\ell(t) \Proj_\ell[\hrefprime\uld{\theta}] + \mathcal{C}\Psi_\ell^0(t)\Proj_\ell[u_0'(\theta_{\text{ref}})\uld{\theta}]\bigr)
\Bigr),
\end{aligned}
\]
whose adjoint is given by
\[
\begin{aligned}
&\Op^*:L^2(0,T;Y^*)\to\Theta^*, \quad \\ &
y^*\mapsto \int_0^T
\sum_{\ell\in\mathbb{N}}
\bigl(\hrefprime^*\Proj_\ell^*\Psi_\ell(t)^*
+u_0'(\theta_{\text{ref}})^*\Proj_\ell^*\Psi_\ell^0(t)^*
\bigr)\mathcal{C}^*y^*(t)\, dt.
\end{aligned}
\]

Let $\dualmap:\Theta\to\Theta^*$ be a single valued selection of the duality map defined by $\langle\dualmap(\theta),\theta\rangle_{\Theta^*,\Theta}=\|\theta\|_\Theta$, $\|\dualmap(\theta)\|_{\Theta^*}=1$, \footnote{existence of $\dualmap$ is a consequence of the Hahn-Banach Theorem; in case of a Hilbert space $\Theta$, $\dualmap$ is the inverse of the Riesz isomorphism} or more generally, a map allowing for the implication 
\begin{equation}\label{dualmap}
\langle \dualmap(\theta),\theta\rangle_{\Theta^*,\Theta} =0 \ \Rightarrow \ \theta=0 \qquad \theta\in\Theta.
\end{equation}
Assuming that
\begin{equation}\label{Xpar_Astar}
\dualmap(\Theta^{par})\subseteq \text{ker}(\Op)^\bot
\end{equation}
holds, and writing \eqref{obs0} as $\uld{\theta}\in\text{ker}(\Op)$, we deduce the chain of implications 
\[
\dualmap(\uld{\theta})\in \text{ker}(\Op)^\bot\text{ and }\uld{\theta}\in\text{ker}{\Op} 
\quad \Rightarrow \quad  
\langle \dualmap(\uld{\theta}),\uld{\theta}\rangle =0 \quad \Rightarrow \quad \uld{\theta}=0.
\]
Since, as another consequence of the Hahn-Banach Theorem, $\overline{\text{ran}(\Op^*)}$ 
is contained in the anihilator of $\text{ker}(\Op)$,\footnote{equality $\overline{\text{ran}(\Op^*)}=\text{ker}(\Op)$ holds, e.g., in the Hilbert space case} a sufficient condition for \eqref{Xpar_Astar} is  
\begin{equation}\label{Xpar_Astar_ran}
\dualmap(\Theta^{par})\subseteq \overline{\text{ran}(\Op^*)}.
\end{equation}

\begin{theorem}\label{thm:lin-uniqueness}
Let $u_{\text{ref}}$ be defined such that \eqref{spacetimeseparable} holds.
Then on any subspace $\Theta^{par}$ of $\Theta$ satisfying \eqref{Xpar_Astar} with $\dualmap$ satisfying \eqref{dualmap}, the linearized forward operator $\mathbb{F}'(u_{\text{ref}},\theta_{\text{ref}}):\Theta^{par}\to L^p(0,T;Y)$ is injective.
\end{theorem}
\begin{remark}\label{rem:Op_reduced}
Note that if $u_{\text{ref}}$ were the solution of \eqref{PDE} with $\theta=\theta_{\text{ref}}$, the operator $\Op$ would coincide with the linearization of the reduced forward operator $F:\theta\mapsto \mathcal{C}u(\theta)$, where $u(\theta)$ solves \eqref{PDE}. However, this would likely not allow to achieve space-time separability \eqref{spacetimeseparable}.
The fact that we do not impose the PDE constraint \eqref{PDE} on $(u_{\text{ref}},\theta_{\text{ref}})$ but use an all-at-once formulation instead, gives us the freedom to choose $u_{\text{ref}}$ such that \eqref{spacetimeseparable} holds.
\end{remark}

\section{Verification for the diffusion equation example}\label{sec:diffusion}
To keep the exposition transparent, we 
first of all
illustrate the abstract theory by the elementary example of identifying a potential $\pot=\pot(x)$ 
and the initial condition $\init=\init(x)$ 
in a diffusion equation subject to several different controls $\contr_n(t)$, 
$n\in\{1,\ldots,\Ncont\}$,
\begin{equation}\label{diffusion}
\begin{aligned}
&u_{n,t}(t,x)-\nabla\cdot (D(x)\nabla u_n(t,x))+\pot(x)\,u(t,x)+\contr_n(t)\,u_n(t,x)=f_n(t,x) \quad (t,x)\in (0,T)\times\Omega\\
&\partial_\nu u_n(t,x)+\gamma u_n(t,x)=0  \quad (t,x)\in (0,T)\times\partial\Omega\\ 
&u_n(0,x)=\init(x), \ x\in \Omega
\end{aligned}
\end{equation}
on a bounded Lipschitz domain $\Omega\subseteq\mathbb{R}^d$, $d\in\{1,2,3\}$,
from either boundary
\begin{equation}\label{C_bndy}
\mathcal{C}v=v\vert_\Gamma \text{ with }\Gamma\subseteq\partial\Omega
\end{equation}
or averaging 
\begin{equation}\label{C_avrg}
\mathcal{C}v=\Bigl(\langle \coilsens^j,v\rangle_{H^1(\Omega)^*,H^1(\Omega)}\Bigr)_{j\in J} \text{ with }\coilsens^j\in H^1(\Omega)^*, j\in J
\end{equation}
observations
\begin{equation}\label{obs_n}
y_n(t)=\mathcal{C}u_n(t)  \qquad t\in(0,T), \quad n\in\{1,\ldots,\Ncont\}.
\end{equation}

Here 
$D\in L^\infty(\Omega;\mathcal{R}^{d\times d})$ is a uniformly positive definite\footnote{not necessarily symmetric; cf. \cite{JiangLiPauronYamamoto2023} and the references therein for some facts on eigenprojections in the nonsymmetric case} 
given coefficient matrix, the source term $f_n\in L^2(0,T;H^1(\Omega)^*)$ is assumed to be known as well, and the equation \eqref{diffusion} is to be understood in the weak $L^2(0,T;H^1(\Omega)^*)$ sense according to the standard textbook framework of, e.g., \cite{EvansBook}. 
Note that we consider control by a bilinear term $p_n\,u_n$, rather than by the source $f_n$, in order to deal with a setting that is closer to the one of the MRI application \eqref{eqn:bloch-torrey}, \eqref{obs_BT}.

With the parameter $\theta=(\pot,\init)$, the state $u=(u_1,\ldots,u_{\Ncont})$, the operators $\DD_t=\partial_t$, $\calA:=-\nabla\cdot (D\nabla \cdot)$ equipped with impedance boundary conditions and acting component wise on $u=(u_1,\ldots,u_{\Ncont})$, the spaces $V=W=H^1(\Omega)^{\Ncont}$, the operators $\BB(t)$ and $\HH(\theta)$ defined as the component wise multiplication operators with $\contr_n\in L^2(0,T)$ and $\pot\in L^2(\Omega)$, respectively,
and $u_0(\theta)=(\init,\ldots,\init)\in L^2(\Omega)^n$,  
this can be cast into the form \eqref{PDE}. 

We choose a space-time separable reference state 
\begin{equation}\label{uref_diffusion}
\begin{aligned}
&u_{\text{ref}}(t,x)=(\psi_1(t),\ldots,\psi_{\Ncont}(t))\,\phi(x)\\
&\text{ with }
\psi_n\in H^1(0,T), \ \phi\in H^1(\Omega)\cap L^\infty(\Omega), \ \tfrac{1}{\phi}\in L^\infty(\Omega),
\end{aligned}
\end{equation} 
so that the operator $\hrefprime$ is defined by $\hrefprime(\uld{\pot},\uld{\init})=\phi\,\uld{\pot}$. 

\subsection{The commutation properties \eqref{commBpsi} and \eqref{commC}} 
\label{subsec:comm_diffusion}
These are trivially satisfied in this setting of  scalar valued state components $u_n$ and controls acting as scalar multipliers; hats on $\BB$, $\mathcal{C}$, $\psi$, $\Psi_\ell$, $\Psi_\ell^0$ can therefore simply be skipped.

\begin{remark}\label{rem:comm_diff}
Considering possible space dependence of $\contr(t)$, commutativity \eqref{commBpsi} would 
entail the conditon 
\begin{equation*}
\forall \ell,\,j\in \mathbb{N}, \ j\not=\ell, \ v_\ell\in \mathbb{E}_\ell,\ v_j\in \mathbb{E}_j \,:\
\langle v_\ell,\contr(t) v_j \rangle=0,
\end{equation*}
see \eqref{commBB_orth} in the proof of Lemma~\ref{lem:comm_BT}.
The following 1-d counterexample shows that this might not be likely to hold. 
On $\Omega=(0,1)$ with vanishing Dirichlet boundary conditions, we have the eigenfunctions $v_\ell(x)=\sin(\ell\pi x)$. Expanding $\contr(t)$ in a Fourier series 
$\contr(t,x)=\sum_{m\in\mathbb{N}_0} \Bigl(s_m(t)\sin(m\pi,x)+c_m(t)\cos(m\pi,x)\Bigr)$, with $s_0=0$, we obtain, using the trigonometric identities
{\small
\[
\begin{aligned}
&\sin \alpha\;\sin \beta\;\sin \gamma=\tfrac{1}{4}\Big(\sin(\alpha+\beta-\gamma)+\sin(\beta+\gamma-\alpha)+\sin(\gamma+\alpha-\beta)-\sin(\alpha+\beta+\gamma)\Big)
\\[0ex]
&\sin \alpha\;\sin \beta\;\cos \gamma=\tfrac{1}{4}\Big(-\cos(\alpha+\beta-\gamma)+\cos(\beta+\gamma-\alpha)+\cos(\gamma+\alpha-\beta)-\cos(\alpha+\beta+\gamma)\Big),
\end{aligned}
\]
}
that 
\[
\begin{aligned}
&\langle v_\ell,\contr(t) v_j \rangle=\\
&\frac14 \sum_{m\in\mathbb{N}_0} \int_0^1\Bigl(
s_m(t)\big(\sin((\ell+j-m)\pi x)+\sin((j+m-\ell)\pi x)\\[-3ex]
&\phantom{\frac14 \sum_{m\in\mathbb{N}_0} \int_0^1\Bigl(s_m(t)\big(}
		  +\sin((m+\ell-j)\pi x)-\sin((\ell+j+m)\pi x)\big)\\[-3ex]
&\phantom{\frac14 \sum_{m\in\mathbb{N}_0} \int_0^1\Bigl(}
+c_m(t)\big(\cos((\ell+j-m)\pi x)+\cos((j+m-\ell)\pi x)\\[-3ex]
&\phantom{\frac14 \sum_{m\in\mathbb{N}_0} \int_0^1\Bigl(+c_m(t)\big(}
		  +\cos((m+\ell-j)\pi x)-\cos((\ell+j+m)\pi x)\big)
\Big)\, dx
\end{aligned}
\]
This yields possible nonzero contributions also in case $j\not=\ell$ for any $m\not=0$:
For $m\in\mathbb{N}$, if $c_m(t)\not=0$, we can choose $j=m+\ell$ to obtain 
$\langle v_\ell,\contr(t) v_j \rangle=c_m(t)\not=0$.
\\
Thus, space dependence of the control multiplier does not seem to be compatible with the commutativity property \eqref{commBpsi} required for the analysis here.
\end{remark}

\subsection{The linear independence condition} 
\label{subsec:linindep_diffusion}
This can be verified by writing the  operator valued functions $\Psi_\ell=(\Psi_{\ell,1},\ldots,\Psi_{\ell,\Ncont})$, $\Psi_\ell^0=(\Psi_{\ell,1}^0,\ldots,\Psi_{\ell,\Ncont}^0)$ according to \eqref{Psi_ell} explicitly as the operators defined by multiplication with the scalar solutions  (with a slight overload of notation)
\begin{equation}\label{PsiPsi0_diffusioneq}
\begin{aligned}
&\Psi_{\ell,n}(t)=-\int_0^t e^{-\lambda_\ell(t-s)-\Contr_n(t)+\Contr_n(s)}\,\psi(s)\, ds
=e^{-\Contr_n(t)}(\tilde{\Psi}_{\ell[,n]}^0*\tilde{\psi}_n)(t)\\
&\Psi_{\ell,n}^0(t)=e^{-\lambda_\ell t-\Contr_n(t)}
=e^{-\Contr_n(t)}\tilde{\Psi}_{\ell[,n]}^0(t)\\
&\text{with }\Contr_n(t)=\int_0^t\contr_n(s)\, ds, \quad
\tilde{\Psi}_{\ell[,n]}^0(t)=e^{-\lambda_\ell t}, \quad
\tilde{\psi}_n(t)=e^{\Contr_n(t)}\psi(t)
\end{aligned}
\end{equation}
where the notation $[,n]$ indicates that $\tilde{\Psi}_{\ell[,n]}^0$ does not actually depend on $n$.
\\
Consider the set $\{\Psi_\ell,\,\Psi_\ell^0\,:\,\ell\in\mathbb{N}\}$ and note that its linear independence is equivalent to the one of $\{\tilde{\Psi}_\ell,\,\tilde{\Psi}_\ell^0\,:\,\ell\in\mathbb{N}\}$ with 
$\tilde{\Psi}_{\ell,n}(t):=e^{\Contr_n(t)}\Psi_{\ell,n}(t)$
$\tilde{\Psi}_{\ell,n}^0(t):=e^{\Contr_n(t)}\Psi_\ell^0(t)$.
By analytic extension to $t\in(0,\infty)$, taking the Laplace transform, we obtain, for all $n\in\{1,\ldots,\Ncont\}$,  
\begin{equation}\label{Laplace_equivalences}
\begin{aligned}
0=\sum_{\ell\in\mathbb{N}} \bigl(\mathfrak{a}_\ell \tilde{\Psi}_{\ell,n}(t)+\mathfrak{c}_\ell \tilde{\Psi}_{\ell[,n]}^0(t)\bigr)
=\sum_{\ell\in\mathbb{N}} \bigl(-\mathfrak{a}_\ell (\tilde{\psi}_n*\tilde{\Psi}_{\ell[,n]}^0)(t)+\mathfrak{c}_\ell \tilde{\Psi}_{\ell[,n]}^0(t)\bigr), \quad t>0, \\
\Leftrightarrow \quad 
0=\sum_{\ell\in\mathbb{N}} \frac{1}{z+\lambda_\ell}\bigl(-\mathfrak{a}_\ell \ (\mathcal{L}\tilde{\psi}_n)(z)+\mathfrak{c}_\ell \bigr), \quad z\in\mathbb{C}\setminus\{-\lambda_k\ : \ k\in\mathbb{N}\}\\
\Leftrightarrow \quad 
0=\sum_{\ell\in\mathbb{N}} \prod_{j\in\mathbb{N}\setminus\{\ell\}}(z+\lambda_j)\bigl(-\mathfrak{a}_\ell (\mathcal{L}\tilde{\psi}_n)(z)+\mathfrak{c}_\ell \bigr)=:\Phi(\tfrac{1}{z}), \quad z\in\mathbb{C}, 
\end{aligned}
\end{equation}
where in the last equivalence we multiplied with 
$\prod_{j\in\mathbb{N}}(z+\lambda_j)$, which allows taking the limits $z\to-\lambda_k$ to extend to all of $\mathbb{C}$. 
Now, setting $z:=\lambda_k$ and using $\prod_{j\in\mathbb{N}\setminus\{\ell\}}(-\lambda_k+\lambda_j)\not=0$, we arrive at
\begin{equation}\label{akLpsick}
\mathfrak{a}_k (\mathcal{L}\tilde{\psi}_n)(\lambda_k)=\mathfrak{c}_k\text{ for all }k\in\mathbb{N}, \ n\in\{1,\ldots,\Ncont\}.
\end{equation}
Note that in case $\mathcal{L}\tilde{\psi}_n$ is an analytic function of $\tfrac{1}{z}$, we even have equivalence of \eqref{Laplace_equivalences} and \eqref{akLpsick}, due to the fact that then also $\Phi$ is analytic and the set $\{\frac{1}{\lambda_k}\,:\,k\in\mathbb{N}\}$ is infinite and has an accumulation point (namely zero).
Choosing $\psi_n$, $\contr_n$ such that with 
$\mathcal{L}\tilde{\psi}_n=(\mathcal{L}\psi_n)*(\mathcal{L}\exp(\int_0^\cdot\contr_n(s)\, ds)$, the matrix {\small $\begin{pmatrix}
(\mathcal{L}\tilde{\psi}_1)(\lambda_k)
&\cdots&(\mathcal{L}\tilde{\psi}_{\Ncont})(\lambda_k)\\
-1&\cdots&-1
\end{pmatrix}$}
has full rank two for each $k\in\mathbb{N}$
that is,
\begin{equation}\label{fullrank}
\forall k\in\mathbb{N} \, \exists n_1=n_1(k),\,n_2=n_2(k)\in\{1,\ldots,\Ncont\}
\,:\ (\mathcal{L}\tilde{\psi}_{n_1})(\lambda_k)\not=(\mathcal{L}\tilde{\psi}_{n_2}),\end{equation}
we achieve linear independence of $\{\Psi_\ell,\,\Psi_\ell^0\,:\,\ell\in\mathbb{N}\}$ and thus \eqref{linearindependence}.

\subsection{The completeness property \eqref{cond:Xpar_gen}} \label{subsec:comp_diffusion}
In case of boundary observations \eqref{C_bndy} with $\Theta^{par}=\Theta=(L^2(\Omega))^2$, this follows from unique continuation, provided 
\begin{equation}\label{uniqueconinuation}
\Sigma\subseteq\overline{\Omega}\subseteq\mathbb{R}^d, \quad d\in\{2,3\}, \quad 
\Sigma\text{ is a $C^1$ manifold}, \quad \text{meas}^{d-1}(\Sigma)>0,
\end{equation}
cf., e.g., \cite{JiangLiPauronYamamoto2023,nonlinear_imaging_JMGT_freq,Tolsa:2023}.
This together with the above mentioned commutation property (subsection~\ref{subsec:comm_diffusion}) and linear independence result  
(subsection~\ref{subsec:linindep_diffusion}) 
implies linearized uniqueness of recovery of $\pot$ and $\init$ from observations \eqref{C_bndy}; uniqueness of their separate reconstruction is well known 
see, e.g., \cite{Isakov:2006}.
\\
In case of averaging observations \eqref{C_avrg} with $\coilsens^j\in L^2(\Omega)$, so that the dual pairing can be written as an $L^2(\Omega)$ inner product and using the $\Proj_\ell$ is the eigenprojection with respect to both $L^2(\Omega)$ and $V=H^1(\Omega)$, we obtain as a sufficient condition 
\begin{equation}\label{Thetapar_pot}
\Theta^{par}\subseteq\overline{\text{span}\{\frac{1}{\phi}\,\Proj_\ell\,\coilsens^j\, : \, \ell\in\mathbb{N}, \, j\in J\}\times\text{span}\{\Proj_\ell\,\coilsens^j\, : \, \ell\in\mathbb{N}, \, j\in J\}}
\end{equation}
for linearized uniquness of $\pot$ and $\init$ under linear independence implied by \eqref{fullrank}. 

\subsection{The general setting \eqref{Xpar_Astar}} 
More generally, from Theorem~\ref{thm:lin-uniqueness} (upon identification of the duality pairing with the inner product via the Riesz isomorphism) with 
\[
\Theta= L^2(\Omega)^2, \quad j(\theta)=\theta, \quad \hrefprime^*w^*=
\begin{pmatrix}\phi\,w^*\\0\end{pmatrix}, \quad
u_0'(\theta_{\text{ref}})^*w^*=\begin{pmatrix}0\\w^*\end{pmatrix},
\]
 we obtain that 
(cf. \eqref{Xpar_Astar})
$
\Theta^{par}\subseteq\overline{\text{ran}(\Op^*)}
$
with 
\begin{equation}\label{Op-adj_diff}
\begin{aligned}
&\Op^*:L^2(0,T;(Z^*)^{\Ncont})^{|J|}\to L^2(\Omega)^2, \quad \\ &
(z_1^*,\ldots,z_{\Ncont}^*)\mapsto 
\sum_{n=1}^{\Ncont}
\int_0^T e^{-\Contr_n(t)}
\sum_{\ell\in\mathbb{N}}
\left(\begin{array}{c}
(\tilde{\psi}_n*e^{-\lambda_\ell \cdot})(t)\,\phi\,\Proj_\ell^*\Psi_\ell(t)^*\mathcal{C}^*z_n^*(t)\\ e^{-\lambda_\ell t}\,\Proj_\ell^*\Psi_\ell(t)^*\mathcal{C}^*z_n^*(t)
\end{array}\right)\, dt
\end{aligned}
\end{equation}
for $\tilde{\psi}_n$, $\Contr_n$ defined as in \eqref{PsiPsi0_diffusioneq},
implies linearized uniqueness.

As already mentioned, under linear independence of the functions 
$\{(\tilde{\psi}_n*e^{-\lambda_\ell \cdot}),\, e^{-\lambda_\ell \cdot}\ : \ \ell\in\mathbb{N}\}$, the condition 
$\Theta^{par}\subseteq\text{ker}(\Op)^\bot=\overline{\text{ran}(\Op^*)}$ 
with \eqref{Op-adj_diff}
can be written as \eqref{Thetapar_pot} in case \eqref{C_avrg}, 
and in case \eqref{C_bndy} is satisfied by the trivial choice 
$\Theta^{par}=L^2(\Omega)$.

Altogether we have proven the following.
\begin{corollary}\label{cor:diffusion}
The forward operator corresponding to the inverse problem of recovering  $(\pot,\init)$ in \eqref{diffusion} from observations \eqref{obs}, linearized at $((\pot_{\text{ref}},\init_{\text{ref}}),\,u_{\text{ref}})$ satisfying \eqref{uref_diffusion} with $\psi_n(t)$ analytic, is injective on 
$\Theta^{par}\subseteq\overline{\text{ran}(\Op^*)}$ with \eqref{Op-adj_diff}.
\\
In particular, if $\contr_n$, $\psi_n$ are chosen such that \eqref{fullrank} holds, then in case of boundary observations \eqref{C_bndy} with \eqref{uniqueconinuation}, linearized uniqueness of $(\pot,\init)$ in $\Theta^{par}=L^2(\Omega)^2$ holds.
The same holds true for averaging observations \eqref{C_avrg} in $\Theta^{par}\subseteq$ \\
$\overline{\text{span}\{\frac{1}{\phi}\,\Proj_\ell\,\coilsens^j\, : \, \ell\in\mathbb{N}, \, j\in J\}\times\text{span}\{\Proj_\ell\,\coilsens^j\, : \, \ell\in\mathbb{N}, \, j\in J\}}$.
\end{corollary}

\begin{remark}\label{rem:diff}
The usefulness of diffusion gets apparent in the definition of spaces in \eqref{cor:diffusion}. Averaging observations in the pointwise in space ODE case $D=0$ would only yield the finite dimensional parameter space 
$\text{span}\{\coilsens^j\,:\, j\in\{1,\ldots,\Ncoil\}\}$ (without the projections) rather than the inifinte dimensional one 
$\overline{\text{span}\{\Proj_\ell\coilsens^j\,:\, j\in\{1,\ldots,\Ncoil\},\,\ell\in\mathbb{N}\}}$.
\\
Indeed, e.g., in the 1-d setting from Remark~\ref{rem:comm_diff}, the theoretical choice 
$\coilsens^1(x)=\sum_{m\in\mathbb{N}} s_m \sin(m\pi x)$ with $s_m\not=0$, for all $m\in\mathbb{N}$ yields  
$\overline{\text{span}\{\Proj_\ell\coilsens^j\,:\, j\in J,\,\ell\in\mathbb{N}\}}\subseteq\overline{\text{span}\{\sin(\ell\pi \cdot)\,:\, \ell\in\mathbb{N}\}}=L^2(\Omega)$,
even with just one time trace $|J|=1$.
\end{remark}

\begin{remark}\label{rem:frac}
The subdiffusion case $\DD_t=\partial_t^\alpha$, $\alpha\in(0,1)$ defined, e.g., by the Djirbashian-Caputo derivative can be covered analogously, basically by substituting exponential by Mittag-Leffler functions. 
\end{remark}

\section{Verification for the Bloch-Torrey example}\label{sec:BT}
Note that in this section we use the vector notation $\vecr$ to distinguish the spatial point $\vecr=(x,y,z)$ from its first component.

Considering the equation in a rotating frame, we can skip the background field $B^0$ (cf., e.g., \cite{MRI-uniqueness}) and thus re-write \eqref{eqn:bloch-torrey} with $B^1$ defined by \eqref{eqn:bloch-torrey-defs} as
\begin{equation}\label{eqn:bloch-torrey-rot}
\begin{aligned}
&    \frac{d}{dt} \vec M(t, \vecr) + \gamma B^1(t, \vecr)\times\vec M(t, \vecr)    
+ R(\vecr) ({\vec M}(t, \vecr)-{\vec M}^0(\vecr))
\\&\hspace*{1cm}
    -\nabla\cdot\Bigl(D(\vecr) \, \nabla \vec M(t, \vecr)\Bigr)
=0\quad (t,\vecr)\in (0,T)\times\Omega, \quad
{\vec M}(0, \vecr)={\vec M}^0(\vecr)\quad \vecr\in \Omega
\end{aligned}
\end{equation}
with $B^1$ and $R$ defined by 
\begin{equation}\label{eqn:bloch-torrey-rot-defs}
\begin{aligned}
&B^1(t, \vecr)=
\begin{pmatrix} p(t)c^+_{x}(\vecr)\\p(t)c^+_{y}(\vecr)\\ 0\skipgradient{\gamma\vec{g}(t)\cdot\vecr}
\end{pmatrix},
\quad 
R(\vecr)=
\begin{pmatrix}
R_\perp(\vecr)&-\gamma\delta B^0(\vecr)&0\\
\gamma\delta B^0(\vecr)&R_\perp(\vecr)&0\\
0&0&R_z(\vecr)
\end{pmatrix}
\end{aligned}
\end{equation}
\skipgradient{
Alternatively, assuming $\vec{g}(t)\equiv \vec{g}_0$, we can use the setting
\begin{equation}\label{eqn:bloch-torrey-rot-defs-g0}
\begin{aligned}
&B^1(t, \vecr)=
\begin{pmatrix} p(t)c^+_{x}(\vecr)\\p(t)c^+_{y}(\vecr)\\0
\end{pmatrix},
\quad 
R(\vecr)=
\begin{pmatrix}
R_\perp(\vecr)&-\gamma(\delta B^0(\vecr)+\vecr \cdot \vec{g}_0)&0\\
\gamma(\delta B^0(\vecr)+\vecr \cdot \vec{g}_0)&R_\perp(\vecr)&0\\
0&0&R_z(\vecr)
\end{pmatrix}
\end{aligned}
\end{equation}
}
To cast the problem into the form \eqref{PDE}, with an operator $\HH$ that is linear with respect to $\theta$ and $u$ we define 
$\vec{u}_n=(M_e,M_x,M_y,M_z)_n$ where we add the trivial PDE for the auxiliary state $M_e$ 
\[
\frac{d}{dt} M_e(t, \vecr)=0, \quad t\in(0,T), \qquad M_e(t, \vecr)=M^{eq}(\vecr)
\]
and set 
\begin{equation}\label{defsBT}
\begin{aligned}
&\theta = (M^{eq},R_z,R_\perp,\delta B^0), \quad
\vec{u}_n=(M_e,M_x,M_y,M_z)_n, \quad n\in\{1,\ldots,\Ncont\}, \quad
u=(\vec{u}_1,\ldots,\vec{u}_{\Ncont}) \\
&\Theta=L^2(\Omega)^4, \quad V=(L^2(\Omega)\times H_0^1(\Omega)^3)^{\Ncont},\\
&\DD_t=\frac{d}{dt}, \quad
\calA=\text{diag}(0,-\Delta_D,-\Delta_D,-\Delta_D), \quad
u_0(\theta)=(M^{eq},0,0,M^{eq})^{\Ncont}\\
&(\BB(t)u(t))(\vecr)
=\gamma \bigl(\ulB_1(t,\vecr)\, \vec{u}_1(t,\vecr),\ldots,\ulB_{\Ncont}(t,\vecr)\, \vec{u}_{\Ncont}(t,\vecr)\bigr)\\
&(\HH(\theta)u(t))(\vecr)=
\bigl(\ulH(\theta(\vecr))\, \vec{u}_1(t,\vecr),\ldots,\ulH(\theta(\vecr))\, \vec{u}_{\Ncont}(t,\vecr)\bigr)\\
&\phantom{(\HH(\theta)u(t))(\vecr)}
=\bigl(\ulQ(\vec{u}_1(t,\vecr))\theta(\vecr),\ldots,\ulQ(\vec{u}_1(t,\vecr))\theta(\vecr)\bigr)
,\\
&\ulB_n(t,\vecr)= 
-p_n(t) \bigl(c_x^+(\vecr)(e_ye_z^T-e_ze_y^T)+c_y^+(\vecr)(e_ze_x^T-e_xe_z^T)\bigr) 
\skipgradient{-\vec{g}_n(t)\cdot\vecr \, (e_xe_y^T-e_ye_x^T)} 
\\
&\ulH(\theta)=\ulH(M^{eq},R_z,R_\perp,\delta B^0)=\text{diag}(0,R_\perp,R_\perp,R_z) +\gamma\delta B^0 (e_ye_x^T-e_xe_y^T)-R_z\,e_z e_e^T\\
&{\footnotesize \ulQ(u)= \ulQ(M_e,M_x,M_y,M_z)= \begin{pmatrix}
0&0&0&0\\
0&0&M_x&-\gamma M_y\\
0&0&M_y&\gamma M_x\\
0&M_z-M_e&0&0
\end{pmatrix}
}.
\end{aligned}
\end{equation}
\skipgradient{(In the setting \eqref{eqn:bloch-torrey-rot-defs-g0}, we just replace $\delta B^0$ by $\delta B^0+\text{id}\cdot\vec{g}_0$ in the definition of $\ulH(\theta)$ and $\vec{g}_n(t)$ by zero in the definition of $\ulB_n(t)$.)
}

We choose a space-time separable reference state 
\begin{equation}\label{uref_BT}
\begin{aligned}
&u_{\text{ref}}(t,x)=(\psi_1(t),\ldots,\psi_{\Ncont}(t))\underline{\phi}(\vecr)\\&\text{ with }\psi_n\in H^1(0,T), \ \underline{\phi}\in H^1(\Omega)^4\cap L^\infty(\Omega)^4, 
\end{aligned}
\end{equation}
so that \eqref{spacetimeseparable} holds with $\hrefprime =\ulQ(\underline{\phi})$.

\medskip

The eigensystem is structured as follows
\begin{equation}\label{eigensystemBT}
\begin{aligned}
&\sigma(\calA+\HH(\theta_{\text{ref}}))=\{0\}\cup\{\lambda_\ell \, :\, \lambda_{\ell}\text{ eigenvalue of }\calA_R
+\gamma\delta B^0_{\text{ref}} (e_ye_x^T-e_xe_y^T)\}, \\ 
&\mathbb{E}_0 = \{(\varphi,0,0, \calA_{R,z}^{-1}[R_{\text{ref},z}\varphi])\, :\,\varphi\in L^2(\Omega)\},
\\ 
&\mathbb{E}_\ell = \{0\}\times\tilde{\mathbb{E}}_{\ell}, \quad
\tilde{\mathbb{E}}_{\ell}
=\text{ker}(\calA_{R,\perp}+\gamma\delta B^0_{\text{ref}} (e_ye_x^T-e_xe_y^T)-\lambda_\ell,\,\calA_{R,z}-\lambda_\ell)\\
&\text{with }\calA_R
=\text{diag}(-\Delta_D+R_{\text{ref},\perp},-\Delta_D+R_{\text{ref},\perp},-\Delta_D+R_{\text{ref},z})\\
&\phantom{\text{with }}
\calA_{R,\perp}=\text{diag}(-\Delta_D+R_{\text{ref},\perp},-\Delta_D+R_{\text{ref},\perp}),\quad
\calA_{R,z}=-\Delta_D+R_{\text{ref},z}, 
\end{aligned}
\end{equation}
where we choose $R_{\text{ref},\perp}$, $R_{\text{ref},z}$ $\in L^\infty(\Omega)$ to be nonnegative functions.
\\
In case $R_{\text{ref},z}=0$, we simply have $\mathbb{E}_0 = L^2(\Omega)\times\{0\}^3$.
\\
In case $\delta B^0_{\text{ref}}=0$, the operator $\calA+\HH(\theta_{\text{ref}})=\text{diag}(0,-\Delta_D+R_{\text{ref},\perp},-\Delta_D+R_{\text{ref},\perp},-\Delta_D+R_{\text{ref},z})$ is symmetric, thus all eigenvalues are real and nonnegative, and the eigenfunctions are complete.
In this setting, due to the eigenvalue identity
\begin{equation}\label{eigvaleq}
(\calA+\HH(\theta_{\text{ref}})) v_\ell =\lambda_\ell v_\ell \ \Leftrightarrow \
\begin{cases}
(-\Delta_D+R_{\text{ref},\perp})v_{\ell,x}=\lambda_\ell v_{\ell,x}\\
(-\Delta_D+R_{\text{ref},\perp})v_{\ell,y}=\lambda_\ell v_{\ell,y}\\
(-\Delta_D+R_{\text{ref},z})v_{\ell,z}=\lambda_\ell v_{\ell,z}
\end{cases}
\end{equation}
one can easily read off that $\lambda_\ell>0$ is an eigenvalue of $\calA+\HH(\theta_{\text{ref}})$ iff it is an eigenvalue of $(-\Delta_D+R_{\text{ref},\perp})$ or of $(-\Delta_D+R_{\text{ref},z})$ with eigenspaces 
\[
\begin{aligned}
&\mathbb{E}_\ell = 
\mathbb{E}_{\perp,\ell}\times\mathbb{E}_{\perp,\ell}\times\mathbb{E}_{z,\ell}
\\
&\text{with } 
\mathbb{E}_{\perp,\ell}=\text{ker}(-\Delta_D+R_{\text{ref},\perp}-\lambda_\ell),\quad
\mathbb{E}_{z,\ell}=\text{ker}(-\Delta_D+R_{\text{ref},z}-\lambda_\ell),
\end{aligned}
\]
where at least one of the spaces in the tensor product defining $\mathbb{E}_\ell$ is nontrivial.
Moreover, the orthogonality relations
\begin{equation}\label{orthogonality}
\begin{aligned}
\forall \ell,\, j\in\mathbb{N}, \ j\not=\ell, \ v_\ell=(v_{\ell,e},\,v_{\ell,x},\,v_{\ell,y},\,v_{\ell,z}) \in \mathbb{E}_\ell, \  
v_j=(v_{j,e},\,v_{j,x},\,v_{j,y},\,v_{j,z}) \in \mathbb{E}_j:\\
\langle v_{\ell,x},\,v_{j,x} \rangle=0, \  
\langle v_{\ell,y},\,v_{j,y} \rangle=0, \
\langle v_{\ell,z},\,v_{j,z} \rangle=0, \
\langle v_{\ell,x},\,v_{j,y} \rangle=0, \
\end{aligned}
\end{equation}
hold (in particular also note the one between $x$ and $y$ components), 
by \eqref{eigvaleq}
and the usual reasoning, e.g., 
$\langle v_{\ell,x},\,v_{j,y} \rangle
=\frac{1}{\lambda_\ell} \langle (-\Delta_D+R_{\text{ref},\perp})v_{\ell,x},\,v_{j,y} \rangle
=\frac{1}{\lambda_\ell} \langle v_{\ell,x},\,(-\Delta_D+R_{\text{ref},\perp})v_{j,y} \rangle
=\frac{\lambda_j}{\lambda_\ell} \langle v_{\ell,x},\,v_{j,y} \rangle$.
\\
If $R_{\text{ref},z}=R_{\text{ref},\perp}$, then additionally to \eqref{orthogonality}
\begin{equation}\label{orthogonality_RperpRz}
\begin{aligned}
\forall \ell,\, j\in\mathbb{N}, \ j\not=\ell, \ v_\ell=(v_{\ell,e},\,v_{\ell,x},\,v_{\ell,y},\,v_{\ell,z}) \in \mathbb{E}_\ell, \  
v_j=(v_{j,e},\,v_{j,x},\,v_{j,y},\,v_{j,z}) \in \mathbb{E}_j:\\
\langle v_{\ell,x},\,v_{j,z} \rangle=0, \
\langle v_{\ell,y},\,v_{j,z} \rangle=0 
\end{aligned}
\end{equation}
holds.

\medskip

\subsection{The commutation properties \eqref{commBpsi} and \eqref{commC}}
\label{subsec:comm_BT}
We first investigate commutativity with the control operator \eqref{commBpsi}.
\begin{lemma}\label{lem:comm_BT}
Let $\BB\in L(V,V)$ be defined by $(\BB v)(\vecr)=\ulB\, v(\vecr)$, $v\in V$, with a (constant) matrix $\ulB\in\mathbb{R}^{4\times4}$, such that $\ulB_{e,\beta}=\ulB_{\beta,e}=0$, $\beta\in\{e,x,y,z\}$, and let, for each $\ell\in\mathbb{N}_0$, the eigenprojection $\Proj_\ell=\text{Proj}_{\mathbb{E}_\ell}$ be defined according to \eqref{eigensystemBT} with $\delta B^0_{\text{ref}}=0$ and $R_{\text{ref},\perp}=R_{\text{ref},z}=0$.

Then for each $\ell\in\mathbb{N}_0$, commutativity
$\Proj_\ell\BB=\BB\,\Proj_\ell$
holds.
\end{lemma}
\begin{proof}
Under the eigensystem structure \eqref{eigensystemBT} with
\[
\mathbb{E}_j=\overline{\text{span}\{\varphi_j^m\,:\,m\in K^m\} }
\]
where $K^0=\mathbb{N}$, due to separability of $L^2(\Omega)$, and for $j\geq1$, $K^j$ is finite, due to compactness of $\calA_R^{-1}$, 
we have equivalence of $\Proj_\ell\BB=\BB\,\Proj_\ell$ with 
\[
\forall i,j\in \mathbb{N}_0, \ m\in K^j, \ k\in K^i \,:\
\underbrace{\langle \varphi_i^k,\Proj_\ell [\ulB \varphi_j^m] \rangle}_{
=\delta_{i,\ell} \langle \varphi_\ell^k,\ulB \varphi_j^m] \rangle}
= \delta_{j,\ell} \langle \varphi_i^k,\hat{\BB}_\ell \varphi_\ell^m \rangle 
\]
By a case distinction, one easily sees that 
this is equivalent to
\begin{equation}\label{commBB_orth}
\forall j\in \mathbb{N}_0\setminus\{\ell\}, \ m\in K^j, \ k\in K^\ell \,:\
\langle \varphi_\ell^k,\ulB \varphi_j^m \rangle=0.
\end{equation}
Written in a component wise manner, this reads as 
\begin{equation}\label{cond_comm_bxbybz}
\begin{aligned}
\forall j\in \mathbb{N}_0\setminus\{\ell\}, \ m\in K^j, \ k\in K^\ell \,:\
& 
\sum_{\alpha\in\{e,x,y,z\}} \sum_{\beta\in\{e,x,y,z\}}
\langle \varphi_{\ell,\alpha}^k,\ulB_{\alpha,\beta} \varphi_{j,\beta}^m \rangle
=0.
\end{aligned}
\end{equation}
In case $\ell=0$ with $\varphi_0^k=(\varphi,0,0, \calA_{R,z}^{-1}[R_{\text{ref},z}\varphi])$ for arbitrary $\varphi\in L^2(\Omega)$, this becomes
\[
\begin{aligned}
\forall j\in \mathbb{N}, \ m\in K^j, \ \varphi\in L^2(\Omega) \,:\
\sum_{\beta\in\{e,x,y,z\}}
\langle \varphi,\ulB_{e,\beta} \varphi_{j,\beta}^m \rangle
+\langle \calA_{R,z}^{-1}[R_{\text{ref},z}\varphi],\ulB_{z,\beta} \varphi_{j,\beta}^m \rangle
=0,
\end{aligned}
\]
that is,
\begin{equation}\label{commcond_ell0}
\begin{aligned}
\forall j\in \mathbb{N}, \ m\in K^j \,:\
\sum_{\beta\in\{e,x,y,z\}}\Bigl(\ulB_{e,\beta} \varphi_{j,\beta}^m
+R_{\text{ref},z}\, \calA_{R,z}^{-1}[\ulB_{z,\beta} \varphi_{j,\beta}^m]\Bigr)
=0 \text{ a.e. in }\Omega
\end{aligned}
\end{equation}
The same condition results with $\ulB$ replaced by $\ulB^T$ in the reverse case $j=0$ and $\ell\in\mathbb{N}$.

Due to the assumed constancy of the matrix entries, as well as the vanishing first row and column, condition \eqref{cond_comm_bxbybz} thus follow from \eqref{orthogonality}, \eqref{orthogonality_RperpRz}, and $R_{\text{ref},z}=0$.
\end{proof}

\begin{remark}\label{rem:comm_BT}
In the particular skew symmetric case 
$\ulB=b_x(e_ye_z^T-e_ze_y^T)+b_y(e_ze_x^T-e_xe_z^T)+b_z(e_xe_y^T-e_ye_x^T)$
relevant here, condition \eqref{cond_comm_bxbybz} (with possibly spatially varying matrix entries) reads as 
\begin{equation*}
\begin{aligned}
\forall j\in \mathbb{N}_0\setminus\{\ell\}, \ m\in K^j, \ k\in K^\ell \,:\
& \langle \varphi_{\ell,x}^k,b_z \varphi_{j,y}^m \rangle
-\langle \varphi_{\ell,x}^k,b_y \varphi_{j,z}^m \rangle
+\langle \varphi_{\ell,y}^k,b_x \varphi_{j,z}^m \rangle\\
&-\langle \varphi_{\ell,y}^k,b_z \varphi_{j,x}^m \rangle
+\langle \varphi_{\ell,z}^k,b_y \varphi_{j,x}^m \rangle
-\langle \varphi_{\ell,z}^k,b_x \varphi_{j,y}^m \rangle
=0.
\end{aligned}
\end{equation*}
A sufficient condition for commutation of $\ulB$ with the eigenprojections, that exploits skew symmetry with possibly nonzero and space dependent $b_z=b_z(\vecr)$ would be
\[
\begin{aligned}
&R_{\text{ref},\perp}=R_{\text{ref},z}= 0, \quad 
\nabla b_x=\nabla b_y=0, \\ 
&\forall \ell,\,j\in \mathbb{N}, j\not=\ell, \ m\in K^j, \ k\in K^\ell \,:\
& \langle \varphi_{\ell,x}^k, b_z \varphi_{j,y}^m \rangle
=\langle \varphi_{\ell,y}^k, b_z \varphi_{j,x}^m \rangle.
\end{aligned}
\]
It is not clear, in which eigenspace setting the latter is possible, though. 
\end{remark}

\medskip

To verify commutativity with the observation operator \eqref{commC} with \eqref{obs_BT}, we use 
\[
(\Psi_\ell(t)v)(\vecr) := \Bigl(\ulPsi_{\ell,n}(t)\,v_n(\vecr)\Bigr)_{n=1}^{\Ncont}
,\quad v\in \mathbb{E}_\ell\subseteq L^2(\Omega)^{4N}
\]
with $\ulPsi_{\ell,n}$ defined in \eqref{Psi_ell_BT} below, 
and the identity 
\[
\begin{aligned}
&\bigl(\mathcal{C}\Psi_\ell(t)v\bigr)_{j;\alpha}
=\Bigl(\langle \coilsens^j_\alpha,(\ulPsi_{\ell,n}(t)v_n)_\alpha\rangle_{L^2(\Omega)}\Bigr)_{n\in\{1,\ldots,\Ncont\}}\\
&\phantom{\bigl(\mathcal{C}\Psi_\ell(t)v\bigr)_{j;\alpha}}
=\Bigl(\sum_{\beta\in\{e,x,y,z\}}
\ulPsi_{\ell,n;\alpha,\beta}(t)\langle \coilsens^j_\alpha,v_{n,\beta}\rangle_{L^2(\Omega)}\Bigr)_{n\in\{1,\ldots,\Ncont\}}
=\bigl(\hat{\Psi}_\ell(t)\hat{\mathcal{C}}v\bigr)_{j;\alpha}\\
&v=(v_1,\ldots,v_{\Ncont})\in L^2(\Omega)^{4N},
\quad j\in\{1,\ldots,{\Ncoil}\},\quad \alpha\in\{x,y\},
\end{aligned}
\] 
with $\hat{\mathcal{C}}:L^2(\Omega)^{4N}\mapsto\bigl(\mathbb{R}^{2\times4}\bigr)^{{\Ncoil}N}$ (independent of $\ell$),
$\hat{\Psi}_{\ell,n}(t):\mathbb{R}^{2\times4}\to \mathbb{R}^2$,
\begin{equation}\label{Chat_BT}
\begin{aligned}
&
\hat{\mathcal{C}} v := 
\left(\langle \coilsens^j_\alpha,v_{n,\beta}\rangle_{L^2(\Omega)}\right)_{\alpha\in\{x,y\},\,\beta\in\{e,x,y,z\},\,j\in\{1,\ldots,{\Ncoil}\},\, n\in\{1,\ldots,\Ncont\}}
,\quad v\in L^2(\Omega)^{4N}, \\
&\hat{\Psi}_{\ell,n}(t) \ul{w} =\Bigl(\sum_{\beta\in\{e,x,y,z\}}
\ulPsi_{\ell,n;\alpha,\beta}(t)\ul{w}_{\alpha,\beta}\Bigr)_{\alpha\in\{x,y\}}, \quad \ul{w}\in \mathbb{R}^{2\times4},
\end{aligned}
\end{equation}
and likewise for $\Psi_{\ell,n}^0$.
\subsection{The linear independence condition}
The key observation here is that due to the definition of the control as multiplication with $4\times 4$ matrix valued functions, the fundamental solutions $\Psi_{\ell,n}(t)$, $\Psi_{\ell,n}^0(t)$ are not just general linear operators on the spaces $\mathbb{E}_\ell$, whose dimension tends to infinity as $\ell\to\infty$, but can be expessed as solutions to (nonautonomous) ODEs governed by $4\times 4$ matrices. 
Therefore, by a proper choice of $\Ncont\geq16$ control functions, linear independence should be realistically achievable. 

According to Lemma~\ref{lem:comm_BT}, in case $\delta B^0_{\text{ref}}=0$, $R_{\text{ref},\perp}=R_{\text{ref},z}=0$, we have, for  
$(\BB_n(t)u(t))(\vecr)=\ulB_n(t,\vecr)\, u(t,\vecr)$
defined by multiplication with a time (but not space) varying matrix $\ulB_n(t)$, that $\Proj_\ell[\BB_n(t)v]=\BB_n(t)\Proj_\ell v$, $\ell\in\mathbb{N}_0$. 

In this situation, 
we define 
for each $\ell\in\mathbb{N}$, the matrix valued (instead of operator valued, cf. \eqref{Psi_ell}) functions $\ulPsi_{\ell,n},\,\ulPsi_{\ell,n}^0:[0,T)\to \mathbb{R}^{4\times 4}$
 as the solutions to the initial value problems  
\begin{equation}\label{Psi_ell_BT}
\begin{aligned}
&\Bigl(\frac{d}{dt}+\lambda_\ell+\gamma p_n(t)\ulC^+\Bigr)\ulPsi_{\ell,n}(t)+\psi_n(t)I_4=0, \ t\in(0,T),\quad \ulPsi_{\ell,n}(0)=0\\  
&\Bigl(\frac{d}{dt}+\lambda_\ell+\gamma p_n(t)\ulC^+\Bigr)\ulPsi_{\ell,n}^0(t)=0, \ t\in(0,T),\quad \ulPsi_{\ell,n}^0(0)=I_4,
\end{aligned}
\end{equation}
with $\ulC^+=c_x^+(\vecr)(e_ye_z^T-e_ze_y^T)+c_y^+(\vecr)(e_ze_x^T-e_xe_z^T)\in\mathbb{R}^{4\times 4}$.

Similarly to \eqref{PsiPsi0_diffusioneq}, we define the auxiliary functions 
\[
\begin{aligned}
&\tilde{\Psi}_{\ell,n}(t)=\exp(\underline{P}_n(t))\ulPsi_{\ell,n}(t), \quad
\tilde{\Psi}_{\ell,n}^0(t)=\exp(\underline{P}_n(t))\ulPsi_{\ell,n}^0(t),\quad
\tilde{\psi}_{n}(t)=\psi_n(t)\exp(\underline{P}_n(t)), \\
&\text{with }\underline{P}_n(t):=\int_0^tp_n(s)\ulC^+\, ds.
\end{aligned}
\]
It is readily checked that the following explicit formula holds (although we will not make use of this here) 
\begin{equation}\label{explicitsolution_BT_DII_A}
\begin{aligned}
&\exp\left(\int_0^tp_n(s)\ulC^+\, ds\right)\\
&=\,U(\vec r)\,\text{diag}\Bigl(1,\,\exp(\imath\gamma|c^+(\vec r)|\int_{t_0}^tp_n(s)\, ds),\,\exp(-\imath\gamma|c^+(\vec r)|\int_{t_0}^tp_n(s)\, ds)\Bigr)\,U(\vec r)^*,
\end{aligned}
\end{equation}
with 
\[
U(\vec r)= \frac{1}{\sqrt{2}|c^+|}\left(\begin{matrix}
c_x^+(\vec r)&\imath c_y^+(\vec r)&\imath c_y^+(\vec r)\\
c_y^+(\vec r)&-\imath c_x^+(\vec r)&-\imath c_x^+(\vec r)\\
0&|c^+(\vec r)|&|c^+(\vec r)|
\end{matrix}\right), \quad 
|c^+|^2=(c_x^+)^2+(c_y^+)^2.
\]
Note that while $\psi_n(t)$ is a scalar, $\tilde{\psi}_{\ell,n}(t)\in\mathbb{R}^{4\times4}$.
\\
This allows us to rewrite 
\eqref{Psi_ell_BT} as 
\begin{equation*}
\begin{aligned}
&\Bigl(\frac{d}{dt}+\lambda_\ell\Bigr)\tilde{\Psi}_{\ell,n}(t)+\tilde{\psi}_n(t)=0, \ t\in(0,T),\quad \tilde{\Psi}_{\ell,n}(0)=0\\  
&\Bigl(\frac{d}{dt}+\lambda_\ell\Bigr)\tilde{\Psi}_{\ell,n}^0(t)=0, \ t\in(0,T),\quad \tilde{\Psi}_{\ell,n}^0(0)=I_4.
\end{aligned}
\end{equation*}

Now, similarly to Section~\ref{sec:diffusion}, for any 
$\underline{\mathfrak{a}}_\ell$, $\underline{\mathfrak{c}}_\ell$ $\in\mathbb{R}^{2\times4}$, $\ell\in\mathbb{N}_0$, we have  
\[
\begin{aligned}
&0=\sum_{\ell\in\mathbb{N}} \bigl( \hat{\Psi}_{\ell,n}(t)\underline{\mathfrak{a}}_\ell+\hat{\Psi}_{\ell,n}^0(t)\underline{\mathfrak{c}}_\ell\bigr), 
\quad t>0
\\
&\Rightarrow \quad 
0=\sum_{\ell\in\mathbb{N}} \bigl( 
\ulPsi_{\ell,n;\alpha,\beta}(t)\underline{\mathfrak{a}}_{\ell;\alpha,\beta}
+\ulPsi_{\ell,n;\alpha,\beta}^0(t)\underline{\mathfrak{c}}_{\ell;\alpha,\beta}\bigr), 
\quad t>0, \ 
\alpha\in\{x,y\}, \ \beta\in\{e,x,y,z\}\\
&\Rightarrow \quad 
0=\sum_{\ell\in\mathbb{N}} \bigl( 
\tilde{\ulPsi}_{\ell,n;\alpha,\beta}(t)\underline{\mathfrak{a}}_{\ell;\alpha,\beta}+\tilde{\ulPsi}_{\ell,n;\alpha,\beta}^0(t)\underline{\mathfrak{c}}_{\ell;\alpha,\beta}\bigr), 
\quad t>0, \
 \alpha\in\{x,y\}, \ \beta\in\{e,x,y,z\},
\end{aligned}
\]
where we have used the fact that $\ulC^+$ only acts on the $x$ and $y$ components. This identity holds for all $n\in\{1,\ldots,\Ncont\}$.
Analogously to \eqref{akLpsick} in Section~\ref{sec:diffusion} we conclude that
under the full rank condition 
\begin{equation}\label{fullrankBT}
\begin{aligned}
&\forall k\in\mathbb{N}, \, \alpha\in\{x,y\}, \, \beta\in\{e,x,y,z\}\, \\
&\exists n_1=n_1(k,\alpha,\beta),\,n_2=n_2(k,\alpha,\beta)\in\{1,\ldots,\Ncont\}
\,: \ 
(\mathcal{L}\tilde{\psi}_{n_1;\alpha,\beta})(\lambda_k)\not=(\mathcal{L}\tilde{\psi}_{n_2;\alpha,\beta})(\lambda_k)
\end{aligned}
\end{equation}
with $\mathcal{L}\tilde{\psi}_n=(\mathcal{L}\psi_n)*(\mathcal{L}\exp\left(\int_0^tp_n(s)\ulC^+\, ds\right)$, 
linear independence of the resulting functions 
$\hat{\ulPsi}_\ell=(\hat{\ulPsi}_{\ell,1},\ldots,\hat{\ulPsi}_{\ell,\Ncont})$, $\hat{\ulPsi}_\ell^0=(\hat{\ulPsi}_{\ell,1}^0,\ldots,\hat{\ulPsi}_{\ell,\Ncont}^0)$ holds. 
Condition \eqref{linearindependence} can therefore be achieved by a proper choice of the reference states via $\psi_n$ and of the controls $p_n$, $n\in\{1,\ldots,\Ncont\}$,  
$\Ncont\geq16$.

\subsection{The completeness property \eqref{cond:Xpar_gen}} 

Motivated by the fact that $\hrefprime$ and $u_0'(\theta_{\text{ref}})$ act on the components of $\uld{\theta}$ in a split manner --- $\hrefprime$ only acts on $\uld{R}_\perp$, $\uld{R}_z$, $\uld{\delta B}^0$, and $u_0'(\theta_{\text{ref}})$ only on $M^{eq}$ -- we define the auxiliary parameter vector
\[
\begin{pmatrix}
\uld{\tilde{\theta}}_0\\
\uld{\tilde{\theta}}_1\\
\uld{\tilde{\theta}}_2\\
\uld{\tilde{\theta}}_3
\end{pmatrix}
:= \begin{pmatrix}
\uld{M}^{eq}\\
\uld{R}_\perp\,\phi_x-\gamma\uld{\delta B}^0\,\phi_y\\
\uld{R}_\perp\,\phi_y+\gamma\uld{\delta B}^0\,\phi_x\\
\uld{R}_z\,(\phi_z-\phi_e)
\end{pmatrix}
=\ulQ(\underline{\phi})
\]
and note that the linear transformation $\mathcal{T}_\phi:\uld{\theta}\to\uld{\tilde{\theta}}$ is bounded and boundedly invertible on $L^2(\Omega)^4$
provided 
\begin{equation}\label{detnonzero}
\begin{aligned}
&d_\phi, \ \tfrac{1}{d_\phi}\, \in L^\infty(\Omega)\text{ for }
d_\phi(\vecr):=\det\bigl(\tilde{\ulQ}(\underline{\phi}(\vecr))\bigr)
=\gamma(\phi_x(\vecr)^2+\phi_y(\vecr)^2)(\phi_z(\vecr)-\phi_e(\vecr))
\\&\text{ with }
\tilde{\ulQ}(\underline{\phi})
=\begin{pmatrix}
0&\phi_x&-\gamma \phi_y\\
0&\phi_y&\gamma \phi_x\\
\phi_z-\phi_e&0&0
\end{pmatrix}.
\end{aligned}
\end{equation}
This implies 
\[
\begin{aligned}
&
(\hrefprime\uld{\theta})_n
=\begin{pmatrix}
0\\
\uld{\tilde{\theta}}_1\\
\uld{\tilde{\theta}}_2\\
\uld{\tilde{\theta}}_3
\end{pmatrix}
=\underline{S}\uld{\tilde{\theta}}
&&
(u_0'(\theta_{\text{ref}})\uld{\theta})_n
=
\begin{pmatrix}
\uld{\tilde{\theta}}_0\\0\\0\\ \uld{\tilde{\theta}}_0
\end{pmatrix} 
=\underline{S}_0\uld{\theta}=\underline{S}_0\uld{\tilde{\theta}}\\
&\text{ with }\underline{S}=\text{diag}(0,1,1,1)
&&
\underline{S}_0=(1,0,0,1)^T (1,0,0,0)
\end{aligned}
\]
according to \eqref{defsBT}. Moreover, with \eqref{Chat_BT}, we have
\[
\hat{\mathcal{C}}_\ellnew v := \hat{\mathcal{C}}_\ellnew^0 v := 
\hat{\mathcal{C}} v := 
\left(\langle \coilsens^j_\alpha,v_\beta\rangle_{L^2(\Omega)}\right)_{\alpha\in\{x,y\},\,\beta\in\{e,x,y,z\}}
,\quad v\in L^2(\Omega)^4. 
\]
Thus, we can write \eqref{cond:Xpar_gen} as
\begin{equation*}
\Theta^{par}\,\cap\,\bigcap_{\ell\in\mathbb{N}_0} 
\text{ker}\bigl(\hat{\mathcal{C}} \Proj_\ell [\ulQ(\underline{\phi})\cdot\,]^{\Ncont}\bigr)\cap 
\text{ker}\bigl(\hat{\mathcal{C}} \Proj_\ell [\underline{S}_0 \cdot\,]^{\Ncont}\bigr)
=\{0\}.
\end{equation*}
or in terms of the transformed coefficients as
\begin{equation*}
\tilde{\Theta}^{par}\,\cap\,\bigcap_{\ell\in\mathbb{N}_0} 
\text{ker}\bigl(\hat{\mathcal{C}} \Proj_\ell [\underline{S} \cdot\,]^{\Ncont}\bigr)\cap 
\text{ker}\bigl(\hat{\mathcal{C}} \Proj_\ell [\underline{S}_0 \cdot\,]^{\Ncont}\bigr)
=\{0\}.
\end{equation*}

We now focus on the setting $R_{\text{ref},\perp}=R_{\text{ref},z}=\delta B^0_{\text{ref}}=0$, 
\skipgradient{$\vec{g}=0$,} 
in which we can guarantee the commutation property \eqref{commBpsi} in case of constant coil sensitivities, and the eigensystem take the simple structure
\[
\begin{aligned}
&\sigma(\calA+\HH(\theta_{\text{ref}}))=\{0\}\cup\sigma(-\Delta_D), \\
&\mathbb{E}_0=L^2(\Omega)\times\{0\}^3, \quad 
\mathbb{E}_\ell=\{0\}\times\text{ker}(-\Delta_D-\lambda_\ell)^3, \ell\in\mathbb{N}.
\end{aligned}
\]
This allows us to more explicitly express the eigenprojections $\Proj_\ell$ to conclude
\[
\begin{aligned}
&\Proj_0 [\ulQ(\underline{\phi})\uld{\theta}]=
\begin{pmatrix}
0\\
0\\
0\\
0
\end{pmatrix}, \quad 
\Proj_\ell [\ulQ(\underline{\phi})\uld{\theta}]=
\begin{pmatrix}
0\\
\Proj_{D,\ell}[\uld{\tilde{\theta}}_1]\\
\Proj_{D,\ell}[\uld{\tilde{\theta}}_2]\\
\Proj_{D,\ell}[\uld{\tilde{\theta}}_3]
\end{pmatrix}
\quad \ell\in \mathbb{N},\\
&\Proj_0 [\underline{S} \uld\theta] = 
\begin{pmatrix}
\uld{\tilde{\theta}}_0\\
0\\
0\\
0
\end{pmatrix}, \quad
\Proj_\ell [\underline{S} \uld\theta] = 
\begin{pmatrix}
0\\
0\\
0\\
\Proj_{D,\ell}[\uld{\tilde{\theta}}_0]
\end{pmatrix}
\quad \ell\in \mathbb{N}.
\end{aligned}
\]
Identifying $\coilsens_\alpha^j$ with the mapping $v\mapsto \langle \coilsens_\alpha^j,v\rangle$ in the sense of the Riesz isomorphism on the Hilbert space $L^2(\Omega)$,
we obtain
\[
\begin{aligned}
&\hat{\mathcal{C}}\,\Proj_0\,[\underline{S} \cdot\,]
=0\\
&\hat{\mathcal{C}}\,\Proj_\ell\,[\underline{S} \cdot\,]
=\Bigl(\text{diag}(0,\Proj_{D,\ell}\coilsens_\alpha^j,\Proj_{D,\ell}\coilsens_\alpha^j,\Proj_{D,\ell}\coilsens_\alpha^j)\Bigr)_{j\in\{1,\ldots,\Ncoil\}, \, \alpha\in\{x,y\}}, 
\quad \ell\in\mathbb{N}\\
&\hat{\mathcal{C}}\,\Proj_0\,[\underline{S}_0 \cdot\,]
=\Bigl(\text{diag}(\coilsens_\alpha^j,0,0,0)\Bigr)_{j\in\{1,\ldots,\Ncoil\}, \, \alpha\in\{x,y\}}\\
&\hat{\mathcal{C}}\,\Proj_\ell\,[\underline{S}_0 \cdot\,]
=\Bigl(\text{diag}(\Proj_{D,\ell}\coilsens_\alpha^j,0,0,0)\Bigr)_{j\in\{1,\ldots,\Ncoil\}, \, \alpha\in\{x,y\}}, 
\quad \ell\in\mathbb{N}.
\end{aligned}
\]
Since $\coilsens_\alpha^j\in\text{span}\{\Proj_{D,\ell}\coilsens_\alpha^j,\,\ell\in\mathbb{N}\}$, thus 
$\{\coilsens_\alpha^j\}^\bot\supseteq
\bigcap_{\ell\in\mathbb{N}}\{\Proj_{D,\ell}\coilsens_\alpha^j\}^\bot$, $\alpha\in\{x,y\}$,
we have 
\[
\begin{aligned}
&\bigcap_{\ell\in\mathbb{N}} 
\text{ker}\bigl(\hat{\mathcal{C}}_\ellnew \Proj_\ell \hrefprime\mathcal{T}_\phi^{-1}\bigr)\cap 
\text{ker}\bigl(\hat{\mathcal{C}}^0_\ellnew \Proj_\ell u_0'(\theta_{\text{ref}})\mathcal{T}_\phi^{-1}\bigr)\\
&=\Bigl(\text{span}\{\Proj_{D,\ell}\coilsens_x^j,\Proj_{D,\ell}\coilsens_y^j\,:\, j\in\{1,\ldots,\Ncoil\},\,\ell\in\mathbb{N}\}^\bot\Bigr)^4,
\end{aligned}
\]
and \eqref{cond:Xpar_gen} reads as 
\[
\Theta^{par}=\mathcal{T}_\phi^{-1}\tilde{\Theta}^{par}
\subseteq \mathcal{T}_\phi^{-1}\Bigl(\Bigl(\overline{\text{span}\{\Proj_{D,\ell}\coilsens_x^j,\Proj_{D,\ell}\coilsens_y^j\,:\, j\in\{1,\ldots,\Ncoil\},\,\ell\in\mathbb{N}\}}\Bigr)^4\Bigr).
\]

\subsection{The general setting \eqref{Xpar_Astar}} 
Also here, we are in a Hilbert space setting and the duality mapping is the identity. Thus linearized uniqueness is implied by 
\[
\mathcal{T}_\phi\Theta^{par}=\tilde{\Theta}^{par}\subseteq \overline{\text{ran}(\Op^*)}
\]
with 
\begin{equation}\label{Op-adj_BT}
\begin{aligned}
&\Op^*:L^2(0,T;\mathbb{R}^{2\Ncont\Ncoil})\to L^2(\Omega)^4, \quad \\ 
&(y_1,\ldots,y_{\Ncont})\mapsto\\ 
&\sum_{j=1}^{\Ncoil}\sum_{n=1}^{\Ncont}
\sum_{\ell\in\mathbb{N}}
\int_0^T 
\Bigl(\text{diag}(0,\Proj_{D,\ell},\Proj_{D,\ell},\Proj_{D,\ell})\ulPsi_{\ell,n}(t)^T
+\text{diag}(I,0,0,0)\ulPsi_{\ell,n}^0(t)^T
\Bigr)
\\&\hspace*{9cm}
(0,y_{n;x}^j(t)\coilsens_x^j,y_{n;y}^j(t)\coilsens_y^j,0)\, dt.
\end{aligned}
\end{equation}
\begin{corollary}\label{cor:BT}
\eqref{detnonzero} 
The forward operator corresponding to the inverse problem of recovering  $(M^{eq},R_z,R_\perp,\delta B^0)$ in \eqref{eqn:bloch-torrey-rot}, \eqref{eqn:bloch-torrey-rot-defs} from observations \eqref{obs_BT}, linearized at $((M^{eq}_{\text{ref}},0,0,0),\,u_{\text{ref}})$ chosen according to \eqref{uref_BT} with $\psi_n(t)$ analytic and $\underline{\phi}$ satisfying \eqref{detnonzero}, is injective on 
$\Theta^{par}\subseteq\overline{\text{ran}(\Op^*)}$ with \eqref{Op-adj_BT}.
\\
In particular, if $p_n$, $\psi_n$ are chosen such that \eqref{fullrankBT} holds, then linearized uniqueness of $(M^{eq},R_z,R_\perp,\delta B^0)$ in $\Theta^{par}\subseteq$ 
$
\mathcal{T}_\phi^{-1}\Bigl(\Bigl(\overline{\text{span}\{\Proj_{D,\ell}\coilsens_x^j,\Proj_{D,\ell}\coilsens_y^j\,:\, j\in\{1,\ldots,\Ncoil\},\,\ell\in\mathbb{N}\}}\Bigr)^4\Bigr)
$ holds.
\end{corollary}

\begin{remark}\label{rem:diff_BT}
Also here, the importance of diffusion for the amount of information obtained from the measurements is essential, cf. Remark~\ref{rem:diff_BT}.
In the MRI context the difference corresponds to model based quantitative imaging with the Bloch-Torrey PDE versus the Bloch ODE.
\end{remark}

\begin{remark}\label{rem:rot}
The explicit form \eqref{Op-adj_BT} of $\Op^*$ in particular shows the importance of at least one of the the pulses $p_n$ rotating the magnetization away from the longitudinal direction (i.e., the $z$ axis) for otherwise, if all $\ulPsi_{\ell,n}(t)$, $\ulPsi_{\ell,n}^0(t)$ are just diagonal matrices, the equilibrium magnetization increment $\uld{M}^{eq}=\uld{\tilde{\theta}}_0$ just gets mapped into zero by $\Op^*$ and is thus not visible to the observations.
\end{remark}

\begin{remark}
An important tool to enhance spatial resolution of the imaging quantities in MRI is the use of a temporally varying gradient field $\vec{g}(t)\in\mathbb{R}^3$ in
\begin{equation}\label{eqn:bloch-torrey-rot-defs-g}
\begin{aligned}
&B^1(t, \vecr)=
\begin{pmatrix} p(t)c^+_{x}(\vecr)\\p(t)c^+_{y}(\vecr)\\ \gamma\vec{g}(t)\cdot\vecr
\end{pmatrix},
\quad 
R(\vecr)=
\begin{pmatrix}
R_\perp(\vecr)&-\gamma\delta B^0(\vecr)&0\\
\gamma\delta B^0(\vecr)&R_\perp(\vecr)&0\\
0&0&R_z(\vecr)
\end{pmatrix}
\end{aligned}
\end{equation}
in place of \eqref{eqn:bloch-torrey-rot-defs} in \eqref{eqn:bloch-torrey-rot}.
This would clearly fit into the above framework \eqref{PDE}, \eqref{obs} and even in most of the particular one from Section~\ref{sec:BT}, upon re-defining 
\[
\ulB_n(t,\vecr)= 
-p_n(t) \bigl(c_x^+(\vecr)(e_ye_z^T-e_ze_y^T)+c_y^+(\vecr)(e_ze_x^T-e_xe_z^T)\bigr) 
-\vec{g}_n(t)\cdot\vecr \, (e_xe_y^T-e_ye_x^T) 
\]
in \eqref{defsBT}.
However, the commutation property can so far not be verified for spatially varying multipliers (and might in fact fail to hold, see the one-dimensional counterexample in Remark~\ref{rem:comm_diff}, as well as Remark~\ref{rem:comm_BT}), which rules out nonvanishing $\vec{g}_n$. 
An alternative way of incorporating a $\vec{g}$ term (at least a time-constant one $\vec{g}(t)=\vec{g}_0$) would be to set 
\begin{equation*}
\begin{aligned}
&B^1(t, \vecr)=
\begin{pmatrix} p(t)c^+_{x}(\vecr)\\p(t)c^+_{y}(\vecr)\\0
\end{pmatrix},
\quad 
R(\vecr)=
\begin{pmatrix}
R_\perp(\vecr)&-\gamma(\delta B^0(\vecr)+\vecr \cdot \vec{g}_0)&0\\
\gamma(\delta B^0(\vecr)+\vecr \cdot \vec{g}_0)&R_\perp(\vecr)&0\\
0&0&R_z(\vecr)
\end{pmatrix},
\end{aligned}
\end{equation*}
remain with \eqref{eqn:bloch-torrey-rot-defs} in \eqref{eqn:bloch-torrey-rot}, and 
just replace $\delta B^0$ by $\delta B^0+\text{id}\cdot\vec{g}_0$ in the definition of $\ulH(\theta)$.
However, also then, the assumption of a vanishing skew symmetric component in the definition of $\calA+\HH(\theta_{\text{ref}})$, that we need in order to conclude orthogonality of the eigenfunctions, would enforce $\vec{g}$ to vanish. 
\end{remark}

\section*{Acknowledgment}
This research was funded in part by the Austrian Science Fund (FWF) 
[10.55776/F100800]. 

\end{document}